\documentclass{amsproc}

\usepackage{amssymb}
\usepackage{graphicx}
\usepackage{amscd}
\usepackage{amsmath}
\theoremstyle{plain}

\newtheorem{definition}{Definition}

\newtheorem{lemma}{Lemma}

\newtheorem{proposition}{Proposition}
\newtheorem{remark}{Remark}

\newtheorem{theorem}{Theorem}

\numberwithin{equation}{section}

\begin{document}

\title[On the Weyl function for complex Jacobi matrices]{On the Weyl function for complex Jacobi matrices}

\author{ A. S. Mikhaylov}
\address{St. Petersburg   Department   of   V.A. Steklov    Institute   of   Mathematics
of   the   Russian   Academy   of   Sciences, 7, Fontanka, 191023
St. Petersburg, Russia and Saint Petersburg State University,
St.Petersburg State University, 7/9 Universitetskaya nab., St.
Petersburg, 199034 Russia.} \email{mikhaylov@pdmi.ras.ru}

\author{ V. S. Mikhaylov}
\address{St.Petersburg   Department   of   V.A.Steklov    Institute   of   Mathematics
of   the   Russian   Academy   of   Sciences, 7, Fontanka, 191023
St. Petersburg, Russia} \email{ftvsm78@gmail.com}

\date{September, 2025}

\maketitle






\noindent {\bf Abstract.} We derive a new representation for the
Weyl function associated with the complex Jacobi matrix in the
finite and semi-infinite cases. In our approach we exploit
connections to the discrete-time dynamical system associated with
these matrices.

\section{Introduction.}

For a given sequence of complex numbers $\{a_1,a_2,\ldots\}$,
$\{b_1, b_2,\ldots  \}$, $a_i\not= 0$, such that $\sup_{n\geqslant
1}\{|a_n|,|b_n|\}\leqslant B$ for some $B>0$, we set
\begin{equation} \label{Jac_matr}
A=\begin{pmatrix} b_1 & a_1 & 0 & 0 & 0 &\ldots \\
a_1 & b_2 & a_2 & 0 & 0 &\ldots \\
0 & a_2 & b_3 & a_3 & 0 & \ldots \\
\ldots &\ldots  &\ldots &\ldots & \ldots &\ldots
\end{pmatrix}.
\end{equation}
For $N\in \mathbb{N}$, by $A^N$ we denote the $N\times N$ Jacobi
matrix which is a block of (\ref{Jac_matr}) consisting of the
intersection of first $N$ columns with first $N$ rows of $A$.

We consider the operator $H$ corresponding to the semi-infinite
Jacobi matrix $A$, defined on $l^2(\mathbb N)\ni \psi=(\psi_1,
\psi_2, \ldots)$, given by the formula
\begin{equation}
\label{Oper_def}
\begin{cases}
(Hu)_n&=a_{n-1}\psi_{n-1}+b_n\psi_n+a_n\psi_{n+1},\quad
n\geqslant 2,\\
(Hu)_1&=b_1\psi_1+a_1\psi_2,\quad n=1 .
\end{cases}
\end{equation}
With $A^N$ we associate the operator $H^N$. The Weyl $m-$functions
for $H$ and $H^N$ are defined (see \cite{BB}) by
\begin{eqnarray}
m(\lambda)=\left( (H-\lambda)^{-1}e_1,e_1\right),\label{Weyl_def1}\\
m^N(\lambda)=\left((H^N-\lambda)^{-1}e_1,e_1\right),
\label{Weyl_def2}
\end{eqnarray}
$e_n=(0,,\ldots,1,0,\ldots)^T$ with $1$ on the $n-$th place.

Note that in the case of a real matrix $A$, according to the
spectral theorem, this definition is equivalent to the following:
\begin{equation}
\label{Weyl_real}
m(\lambda)=\int_{\mathbb{R}}\frac{d\rho(z)}{z-\lambda},
\end{equation}
where $d\rho=d\left(E_He_1,e_1\right)$ and $E_H$ is the
(projection-valued) spectral measure of the self-adjoint operator
$H$.

Associated with the matrix $A$ and additional parameter
$a_0\not=0$ is a dynamical system with discrete time:
\begin{equation}
\label{Jacobi_dyn}
\begin{cases}
u_{n,t+1}+u_{n,t-1}-a_{n}u_{n+1,t}-a_{n-1}u_{n-1,t}-b_nu_{n,t}=0,\quad n,t\in \mathbb{N},\\
u_{n,-1}=u_{n,0}=0,\quad n\in \mathbb{N}, \\
u_{0,t}=f_t,\quad t\in \mathbb{N}\cup\{0\},
\end{cases}
\end{equation}
which is a natural analog of a dynamical systems governed by a
wave equation on a semi-axis \cite{AM,BM_1}. By an analogy with
continuous problems \cite{B07}, we treat the complex sequence
$f=(f_0,f_1,\ldots)$ as a \emph{boundary control}. The solution to
(\ref{Jacobi_dyn}) is denoted by $u^f_{n,t}$. We consider also the
dynamical system associated with the finite matrix $A_N$:
\begin{equation}
\label{Jacobi_dyn_int}
\begin{cases}
v_{n,t+1}+v_{n,t-1}-a_nv_{n+1,t}-a_{n-1}v_{n-1,t}-b_nv_{n,t}=0,\,\, t\in \mathbb{N}_0,\,\, n\in 1,\ldots, N,\\
v_{n,\,-1}=v_{n,\,0}=0,\quad n=1,2,\ldots,N+1, \\
v_{0,\,t}=f_t,\quad v_{N+1,\,t}=0,\quad t\in \mathbb{N}\cup\{0\},
\end{cases}
\end{equation}
which is a natural analog of a dynamical systems governed by a
wave equation on an interval, the solution to
(\ref{Jacobi_dyn_int}) is denoted by $v^f$.

Having fixed $T\in \mathbb{N}$, we associate the \emph{response
operators} with (\ref{Jacobi_dyn}), (\ref{Jacobi_dyn_int}) which
act according to the rules:
\begin{eqnarray}
\left(R^T f\right)_t:=u^f_{1,t},\quad t=1,\ldots, T,\label{Resp_op}\\
\left(R^T_N f\right)_t:=v^f_{1,t},\quad t=1,\ldots, T.
\end{eqnarray}

In the second section we recall some results on complex Jacobi
matrices in our case in accordance with \cite{BB}. We also provide
some results from \cite{SG} on the representation of resolvent
$(H-\lambda I)^{-1}$.

In the third section we derive the special "spectral
representation" to the solution of (\ref{Jacobi_dyn_int}).

In the final section, we  we introduce the discrete Fourier
transform from \cite{MMS}  and derive representations for the Weyl
functions for $H$ and $H^N$.

\section{Operators associated with complex Jacobi matrices and their Weyl functions. }

With the matrix $A$ we can associate operator $H$ acting in the
linear space of vectors from $l_2$ with a finite number of nonzero
elements, using the usual matrix product.

The \emph{maximal operator} $H_{max}$ associated with the matrix
$A$ is defined on the domain
\begin{equation*}
D(H_{max}):=\left\{y\in l_2\,|\, Ay\in l_2\right\}.
\end{equation*}
Another operator associated with $A$ is a \emph{minimal operator}
$H_{min}$, which by definition is a smallest closed extension of
$A$.
\begin{definition}
The matrix $A$ is called \emph{proper} if the operators $H_{max}$
and $H_{min}$ coincide.
\end{definition}

We consider two solution to the difference equation
\begin{equation}
\label{Dif_eq} H\phi=z\psi,
\end{equation}
with initialization
\begin{equation*}
q_0(z)=1,\quad q_{1}(z)=0,\quad p_0(z)=0,\quad p_{1}(z)=1.
\end{equation*}

\begin{definition}
The complex Jacobi matrix is called determinate if at least one of
the sequences $\left\{p_k(0)\right\}$, $\left\{q_k(0)\right\}$ is
not an element of $l_2$.
\end{definition}

In the case of finite Jacobi matrix by $\phi^+(z)$ we define the
solution of (\ref{Dif_eq}) with Cauchy data "at the right end of
the interval"
\begin{equation*}
\phi^+_{N+1}=0,\quad \phi^+_{N}=1.
\end{equation*}
For the case of semi-infinite $A$ we know \cite[Section 2.1]{BB}
that the bounded matrix is proper. Moreover, by \cite[Theorem
2.6]{BB}, we have that the proper matrix is determinate and by
\cite[Theorem 2.8]{BB} for $z\in \mathbb{C}\backslash
\sigma_{ess}(H)$ there exists $l_2$ solution to the equation
(\ref{Dif_eq}), which we also denote by $\phi^+(z)$. Note that
since our matrix is bounded, $\sigma_{ess}(H)$ is inside the ball
of radius $B$.

For two solution $u, v$ of (\ref{Dif_eq}) we introduce the
\emph{Wronskian} $W(u,b)$ by the rule
\begin{equation*}
W(u,v)[n]:=a_n\left(u_nv_{n+1}-u_{n+1}v_n\right).
\end{equation*}
Simple observation (as in the real case) is that $W$ does not
depends on $n$. Indeed, for two solutions
\begin{gather*}
a_{n-1} u_{n-1}+b_n u_n+a_n u_{n+1}=zu_n,\\
a_{n-1} v_{n-1}+b_n v_n+a_n v_{n+1}=zv_n,
\end{gather*}
we multiply the first line by $v_n$ and the second line by $u_n$
and subtract second from the first, to get
\begin{equation*}
a_{n-1}(u_{n-1}v_n-v_{n-1}u_n)=a_n(u_{n}v_{n+1}-v_{n}u_{n+1}),
\end{equation*}
thus we have that
\begin{equation*}
W(u,v):=W(u,v)[n],\quad n=1,2,\ldots.
\end{equation*}
In what follows we need to define $\phi^+_0$, to do so we formally
set $a_0=1$, then $\phi^+_0$ can be determined (cf. second line in
(\ref{Oper_def})) from the equality
\begin{equation*}
a_0\phi^+_0+b_1\phi^+_1+a_1\phi^+_2=z\phi^+_1.
\end{equation*}
As such we have that $W(u,v)=W(u,v)[0]$.

The following lemma is a complex analog of \cite[Propositions 2.2,
2.3]{SG}.
\begin{lemma}
The Weyl functions for $H$ and $H^N$ have the representation
\begin{equation}
\label{m_funk} m(z)=-\frac{\phi^+_1(z)}{\phi^+_0(z)}.
\end{equation}
\end{lemma}
\begin{proof}
The Green function (the kernel of resolvent) is defined by
\begin{equation*}
G_{m,n}(z)=\left((H-z)^{-1}e_m,e_n\right).
\end{equation*}
Then it is straightforward to check that
\begin{equation*}
G_{m,n}(z)=\frac{1}{W(p,\phi^+)}p_{\operatorname{min}\{m,n\}}(z)\phi^+_{\operatorname{max}\{m,n\}},
\end{equation*}
where the choice of $p,$ $\phi^+$ guarantees the equation holds at
$n=1$ and $n=N$ in finite case and $n=1$ and that
$\left\{\sum_nG_{m,n}(z)f_n\right\}_{m=1}^\infty\in l_2$ for any
$f\in l_2$ with finite number of nonzero elements.

Then by the definition (\ref{Weyl_def1}), (\ref{Weyl_def2}),
choice of $a_0=1$ and conditions on $p$ at $n=0,1$, we have that
in finite and semi-infinite cases
\begin{equation*}
m(z)=G_{1,1}(z)=\frac{p_1(z)\phi^+_1(z)}{a_0((p_{0}(z)\phi^+_{1}(z)-p_{1}(z)\phi^+_{0}(z)))}=-\frac{\phi^+_1(z)}{\phi^+_0(z)},
\end{equation*}
which completes the proof.
\end{proof}

\section{Special representation of the solution to (\ref{Jacobi_dyn_int})}

Here we outline the derivation of the special representation of
the the solution to (\ref{Jacobi_dyn_int}) according to
\cite{MM6}. First we need the Autonne-Takagi \cite{W}
factorization:
\begin{theorem}
Let $H\in \mathbb{C}^{n\times n}$ be a complex symmetric matrix:
$H^*=\overline H$, then there exists a unitary matrix $U$ such
that
\begin{equation}
\label{AutTak}
UHU^T=D=\begin{pmatrix} \hat{d}_1& 0&\ldots &0\\
0& \hat{d}_2&\ldots &0\\
 \ldots& \ldots&\ldots &\ldots\\
 0& 0&\ldots &\hat{d}_n
\end{pmatrix},
\end{equation}
where $\hat{d}_i\geqslant 0,$ $i=1\ldots,n$.
\end{theorem}


We use this theorem for the matrix $A^N$: in \cite{MM6} it is
shown that one can choose unitary $U$ (we drop $N$) such that
\begin{equation*}
UA^N\left(U\right)^T=\begin{pmatrix} \hat{d}_1& 0&\ldots &0\\
0& \hat{d}_2&\ldots &0\\
 \ldots& \ldots&\ldots &\ldots\\
 0& 0&\ldots &\hat{d}_n
\end{pmatrix},\quad \hat{d}_i\in \mathbb{C}\backslash\{0\}, \hat{d}_i\not=\hat{d}_j,\, i\not=j,\,i,j=1\ldots,N.
\end{equation*}

If we introduce the notation
\begin{equation*}
U=\begin{pmatrix} \hat{u}^1\\
\hat{u}^2\\
\ldots\\
\hat{u}^N
\end{pmatrix},\quad U^T=\left(\hat{u}^1 \,|\, \hat{u}^2 \,|\, \ldots \,|\,
\hat{u}^N\right),
\end{equation*}
that is $\hat{u}^i$ is a line in the first equality or a column in
the second, then $\hat{u}^i$ satisfies:
\begin{equation}
\label{Spectral_strange}
A^N\hat{u}^i=\hat{d}_i\overline{\hat{u}^i},\quad A^N\begin{pmatrix} \hat{u}^i_1\\
{\hat{u}}^i_2\\
\ldots\\
{\hat{u}}^i_n
\end{pmatrix}=\hat{d}_i \begin{pmatrix}\overline{\hat{u}^i_1}\\
\overline{{u}^i_2}\\
\ldots\\
\overline{\hat{u}^i_n}
\end{pmatrix}.
\end{equation}
Note that first components of all vectors is not zero:
\begin{equation*}
\hat{u}^i_1\not=0,\quad i=1,\ldots,N,
\end{equation*}
otherwise it immediately follows from (\ref{Spectral_strange})
that $\hat{u}^i=0$.

Now we introduce the vectors which we use in Fourier-type
expansion for the solution (\ref{Jacobi_dyn_int}) in the following
way:
\begin{equation*}
u^i=\begin{pmatrix} 0\\
\frac{\hat{u}^i_1}{\hat{u}^i_1}\\
\ldots\\
\frac{\hat{u}^i_N}{\hat{u}^i_1}\\
0
\end{pmatrix},\quad i=1,\ldots,N,
\end{equation*}
so we formally add two values: $u^i_0=u^i_{N+1}=0$ and normalize
vectors such that $u^i_1=1$. In this case we have (see
(\ref{Spectral_strange})
\begin{equation*}
A^N\begin{pmatrix} \frac{\hat{u}^i_1}{\hat{u}^i_1}\\
\frac{\hat{u}^i_2}{\hat{u}^i_1}\\
\ldots\\
\frac{\hat{u}^i_n}{\hat{u}^i_1}
\end{pmatrix}={d}_i \begin{pmatrix}\frac{\overline{\hat{u}^i_1}}{\overline{\hat{u}^i_1}}\\
\frac{\overline{\hat{u}^i_2}}{\overline{\hat{u}^i_1}}\\
\ldots\\
\frac{\overline{\hat{u}^i_n}}{{\overline{\hat{u}^i_1}}}
\end{pmatrix},
\end{equation*}
where we introduced the notation
\begin{equation*}
d_i:=\hat{d_i}\frac{\overline{\hat{u}^i_1}}{{\hat{u}^i_1}},\quad
i=1,\ldots,N.
\end{equation*}


Now we look for the solution to (\ref{Jacobi_dyn_int}) in the
form:
\begin{equation}
\label{Repr2}
v^f_{n,t}=\begin{cases} \sum_{k=1}^N c^k_t\overline {u^k_n},\\
f_t,\quad n=0.
\end{cases}
\end{equation}
Introducing the notations
\begin{eqnarray}
\sum_{n=1}^N \overline{u^k_n}u_n^i=\delta_{ki}\rho_i,\quad i=1,\ldots,N,\label{Rho}\\
H_{ki}=\sum_{n=1}^N \overline{u^k_n}\overline{u_n^i}, \quad
k,i=1,\ldots,N,\label{HKL}
\end{eqnarray}
We see that $c^i_t$ are determined from
\begin{equation}
\label{Repr4} c_{t+1}^i+c_{t-1}^i -\frac{d_i}{\rho_i}\sum_{k=1}^N
c^k_tH_{ki}=\frac{a_0}{\rho_i}f_t,\quad i=1,\ldots,N.
\end{equation}
We look for the solution to (\ref{Repr4}) in the form:
\begin{equation}
\label{CT_coeff}
c_t^i=\frac{a_0}{\rho_i}\sum_{l=0}^tf_lT_{t-l}^{(i)}.
\end{equation}
We introduce the notation
\begin{equation}
\label{Omega} \omega_i=\sum_{k=1}^N d_i\frac{H_{ki}}{\rho_k},\quad
i=1,\ldots,N,
\end{equation}
then $T^{(i)}_t$ satisfies
\begin{equation}
\label{T_coeff}
\begin{cases}
T_{t+1}^{(i)}+T_{t-1}^{(i)}-\omega_iT_t^{(i)}=0,\\
T_0^{(i)}=0,\,\, T_{-1}^{(i)}=-1.
\end{cases}
\end{equation}
Or, in other words, $T_t^{(i)}$ are simply the Chebyshev
polynomials of the second kind evaluated at points $\omega_i$:
$T^{(i)}_t=T_t(\omega_i)$.

We fix some positive integer $T$ and denote by $\mathcal{F}^T$ the
\emph{outer space} of the system (\ref{Jacobi_dyn_int}), the space
of controls: $\mathcal{F}^T:=\mathbb{C}^T$, $f\in \mathcal{F}^T$,
$f=(f_0,\ldots,f_{T-1})$, $f,g\in \mathcal{F}^T$,
$(f,g)_{\mathcal{F}^T}=\sum_{k=0}^{T-1} f_k\overline{g_k}$. And
let $\mathcal{F}^\infty=\left\{(f_0,f_1,\ldots)\,|\, f_i\in
\mathbb{C},\, i=0,1,\ldots  \right\}$, so $\mathcal{F}^\infty$ is
the set of complex sequences.

\begin{definition}
For $f,g\in \mathcal{F}^\infty$ we define the convolution
$c=f*g\in \mathcal{F}^\infty$ by the formula
\begin{equation*}
c_t=\sum_{s=0}^{t}f_sg_{t-s},\quad t\in \mathbb{N}\cup \{0\}.
\end{equation*}
\end{definition}

\begin{definition}
The \emph{response operator} $R^T:\mathcal{F}^T\mapsto
\mathbb{C}^T$ for the system (\ref{Jacobi_dyn}) is defined by
(\ref{Resp_op})
\end{definition}
The \emph{response vector} is the convolution kernel of the
response operator,
$r=(r_0,r_1,\ldots,r_{T-1})=(a_0,w_{1,1},w_{1,2},\ldots
w_{1,T-1})$, in \cite{MM4}  it is shown that $R$ is a convolution
operator:
\begin{equation*}
\left(R^Tf\right)_t=u^f_{1,t}=r*f_{\cdot-1}.
\end{equation*}
 By choosing the special control
 $f=\delta=(1,0,0,\ldots)$, the kernel of the response operator can be determined as
\begin{equation*}
\left(R^T\delta\right)_t=u^\delta_{1,t}=
r_{t-1},\quad t=1,2,\ldots.
\end{equation*}

Thus for special control $f=\delta$, using (\ref{Repr2}),
(\ref{CT_coeff}), (\ref{T_coeff}) one have that:
\begin{equation}
\label{Repr6} r_{t-1}^N=v^\delta_{1,t}= \sum_{k=1}^N
c^k_t\overline {u^k_1}=\sum_{k=1}^N
c^k_t=\sum_{k=1}^N\frac{1}{\rho_k}T_t(\omega_k),\quad
t=1,2,\ldots,
\end{equation}
where $\rho_k$ and $\omega_k$ are defined in (\ref{Rho}) and
(\ref{Omega}).

We introduce a discrete measure $d\rho^N$ on $\mathbb{C}$,
concentrated on the set of points
$\left\{\omega_k\right\}_{k=1}^N$, by definition we set
\begin{equation*}
d\rho^N(\{\omega_k\})=\frac{1}{\rho_k},
\end{equation*}
so that at points $\omega_k$ it has weights $\frac{1}{\rho_k}$.
Then we can rewrite (\ref{Repr6}) in a form which resembles the
spectral representation of dynamic inverse data (see \cite{MM5}).

Due to the finite speed of wave propagation in (\ref{Jacobi_dyn})
the solution $u^f$ depends on the coefficients $a_n,b_n$ as
follows:
\begin{remark}
\label{Rem1} For $M\in \mathbb{N}$, $u^f_{M-1,M}$ depends on
$\{a_0,\ldots,a_{M-1}\}$, $\{b_1,\ldots,b_{M-1}\}$, the response
$R^{2T}$ (or, what is equivalent, the response vector
$(r_0,r_1,\ldots,r_{2T-2})$) depends on $\{a_0,\ldots,a_{T-1}\}$,
$\{b_1,\ldots,b_T\}$.
\end{remark}
This, in particular means that for responses of the systems
(\ref{Jacobi_dyn}), (\ref{Jacobi_dyn_int}) the following relation
holds:
\begin{equation}
\label{resp_fin_inf} r_t=r^N_t,\quad t=1,2,\ldots, 2N-2.
\end{equation}

The following proposition is one of results from \cite{MM6}:
\begin{proposition}
The dynamic response vector of the system (\ref{Jacobi_dyn_int})
admits the following representation:
\begin{equation}
\label{Resp_vect_N}
r_{t-1}^N=\int_{\mathbb{C}}T_t(\lambda)\,d\rho^N(\lambda),\quad
t=1,2,\ldots.
\end{equation}
The dynamic response vector of the system (\ref{Jacobi_dyn})
admits the representation:
\begin{equation}
\label{Resp_vect_N1}
r_{t-1}=\int_{\mathbb{C}}T_t(\lambda)\,d\rho(\lambda),\quad
t=1,2,\ldots.
\end{equation}
\end{proposition}
\begin{proof}
Formula (\ref{Resp_vect_N}) is a consequence of  (\ref{Repr6}) and
definition of $d\rho^N(\lambda)$. From (\ref{resp_fin_inf}) it
follows that $r_t=\int_{\mathbb{C}}T_t(\lambda)\,d\rho^N(\lambda)$
with any $N$ as soon as $t\leqslant 2N-2.$ Then for arbitrary
polynomial $P\in C[\lambda]$,
$P(\lambda)=\sum_{k=0}^Ma_k\lambda^k$, we have that
\begin{equation}
\label{measure_conv}
\int_{\mathbb{C}}P(\lambda)\,d\rho^N(\lambda)\longrightarrow_{N\to\infty}
\sum_{k=0}^Ma_ks_k,\quad
s_k=\int_{\mathbb{C}}\lambda^k\,d\rho^N(\lambda).
\end{equation}
Convergence (\ref{measure_conv}) means that $d\rho^N$ converges
$*-$weakly as $N\to\infty$ to certain measure $d\rho$, which
provides (\ref{Resp_vect_N1}).
\end{proof}

\section{Representation of Weyl functions. }

We need to introduce the appropriate Fourier transformation to
pass from "dynamic" problem with time $t$ to "spectral" problem
with parameter $\lambda$. Here we follow \cite{MMS}:

\subsection{Discrete Fourier transformation.}

Consider two solutions of the equation
\begin{equation}
\label{Difer_eqn0} \psi_{n-1}+\psi_{n+1}=\lambda \psi_n,
\end{equation}
satisfying the initial conditions:
\begin{equation*}
P_0=0,\,\,P_1=1,\quad Q_0=-1,\,\, Q_1=0.
\end{equation*}
Clearly, one has:
\begin{equation*}
Q_{n+1}(\lambda)=P_{n}(\lambda).
\end{equation*}
We are looking for the unique $l^2$-solution $S$:
\begin{equation*}
S_n(\lambda)=Q_n(\lambda)+m_0(\lambda)P_n(\lambda),
\end{equation*}
where $m_0$ is the Weyl function of $H_0$ (the special case of $H$
with  $a_n\equiv1$, $b_n\equiv0$).
Consider the new variable $z$ related to $\lambda$ by the
equalities
\begin{equation}
\label{Lam_z_conn} \lambda=z+\frac1z, \quad
z=\frac{\lambda-\sqrt{\lambda^2-4}}2=\frac{\lambda-i\sqrt{4-\lambda^2}}2,
\end{equation}
$z\in\mathbb D\cap\mathbb C_-$ for $\lambda\in\mathbb C_+$, where
$\mathbb{D}=\{z\,|\, |z|\leqslant 1\}$. It is known \cite{MMS}
that
$$
m_0(\lambda)=-z.
$$
Equation \eqref{Difer_eqn0} has two solutions, $z^n$ and $z^{-n}$.
Since $S\in l^2(\mathbb N)$ and $S_1(\lambda)=m_0(\lambda)=-z$, we
get
$$
S_n(\lambda)=-z^n.
$$
Let $\chi_{(a,b)}(\lambda)$ be the characteristic function of the
interval $(a,b)$. The spectral measure of the unperturbed operator
$H_0$ corresponding to (\ref{Difer_eqn0}) is
\begin{equation*}
d\rho(\lambda)=\chi_{(-2,2)}(\lambda)\frac{\sqrt{4-\lambda^2}}{2\pi}\,d\lambda.
\end{equation*}
Then the Fourier transformation $F: l^2(\mathbb{N})\mapsto
L_2((-2,2),\rho)$ acts by the rule: for $(v_n)\in
l^2(\mathbb{N})$:
\begin{equation*}
F(v)(\lambda)=V(\lambda)=\sum_{n=0}^\infty S_n(\lambda)v_n.
\end{equation*}
The inverse transform is given by:
\begin{equation*}
v_n=\int_{-2}^2V(\lambda)S_n(\lambda)\,d\rho(\lambda).
\end{equation*}

\subsection{Main result}

The following theorem is a complex analogue of the theorem from
\cite{MMS}. The strategy again consists of using the Fourier
transform to go from dynamic systems (\ref{Jacobi_dyn}),
(\ref{Jacobi_dyn_int}) to "spectral" systems with parameter
$\lambda$. The difference lies in the proof for the semi-infinite
case. In the case of real $A$ we used the convergence of the Weyl
functions corresponding to the blocks $A^N$ to Weyl function for
the semi-infinite matrix, here we do not have this option, and we
need to show directly that the corresponding solution belongs to
$l_2$.

\begin{theorem}
\label{Theor_main} If the coefficients in the semi-infinite $H$ or
finite $H^N$ Jacobi matrix operators satisfy $\sup_{n\geqslant
1}\{a_n,|b_n|\}\leqslant B$, then the Weyl functions $m,$ $m^N$
admit the representations
\begin{eqnarray}
m(\lambda)=-\sum_{t=0}^\infty z^tr_t,\label{M_inf}\\
m^N(\lambda)=-\sum_{t=0}^\infty z^tr^N_t\label{M_fin}
\end{eqnarray}
in terms of response vectors (kernels of dynamic response
operators associated to (\ref{Jacobi_dyn}),
(\ref{Jacobi_dyn_int})): $\left(r_0,r_1,r_2,\ldots\right)$,
$\left(r_0^N,r_1^n,r_2^N,\ldots\right)$,  where the variables
$\lambda$ and $z$ are related by
\begin{equation*}
\lambda=z+\frac1z, \quad
z=\frac{\lambda-\sqrt{\lambda^2-4}}2=\frac{\lambda-i\sqrt{4-\lambda^2}}2.
\end{equation*}
These representations hold for $\lambda\in D$, where $D\subset
\mathbb{C}_+$  is defined as follows. Let $R:=3B+1$, then
\begin{eqnarray}
\label{D_reg} D:=\left\{x+iy\,\Big|\, x>
\left(R+\frac{1}{R}\right)\cos{\phi},\,
y>\left(\frac{1}{R}-R\right)\sin{\phi},\,\phi\in(\pi,2\pi)\right\}.\notag
\end{eqnarray}
\end{theorem}

\begin{proof}[ Proof of Theorem 1.]
In (\ref{Jacobi_dyn}), (\ref{Jacobi_dyn_int}) we take the special
controls: $f=\delta$, $g=\delta$ and go over the Fourier transform
in the variable $t$: for a fixed $n$ we evaluate the sum using the
conditions at $t=0$:
\begin{equation*}
\sum_{t=0}^\infty
S_t(\lambda)\left[u^\delta_{n,t+1}+u^\delta_{n,t-1}-a_nu^\delta_{n+1,t}-a_{n-1}u^\delta_{n-1,t}-b_nu^\delta_{n,t}\right]=0.
\end{equation*}
Changing the order of summation yields:
\begin{equation}
\label{Sum} \sum_{t=0}^\infty
u^\delta_{n,t}\left[S_{t-1}(\lambda)+S_{t+1}(\lambda)\right]-
S_t(\lambda)\left[a_nu^\delta_{n+1,t}+a_{n-1}u^\delta_{n-1,t}+b_nu^\delta_{n,t}\right]=0.
\end{equation}
Introducing the notation
\begin{equation}
\label{U_hat} \widehat u_n(\lambda):=\sum_{t=0}^\infty
S_t(\lambda)u^\delta_{n,t}
\end{equation}
and assuming that $S$ satisfies (\ref{Difer_eqn0}), we deduce from
(\ref{Sum}) that $\widehat u$ satisfies
\begin{equation}
\label{Syst_lambda_infinit} \left\{
\begin{array}l
a_{n-1}\widehat u_{n-1}(\lambda)+a_n\widehat
u_{n+1}(\lambda)+b_n\widehat
u_{n}(\lambda)=\lambda\widehat u_n(\lambda),\\
\widehat u_0(\lambda)=-1.
\end{array}
\right.
\end{equation}
Similarly, introducing
\begin{equation}
\label{V_hat} \widehat v_n(\lambda):=\sum_{t=0}^\infty
S_t(\lambda)v^{\delta}_{n,t}
\end{equation}
one can check that $\widehat v$ satisfies
\begin{equation}
\label{Syst_lambda_finit} \left\{
\begin{array}l
a_n\widehat v_{n+1}(\lambda)+a_{n-1}\widehat
v_{n-1}(\lambda)+b_n\widehat
v_{n}(\lambda)=\lambda\widehat v_n(\lambda),\\
\widehat v_0(\lambda)=-1, \,\, \widehat v_{N+1}=0.
\end{array}
\right.
\end{equation}
Then by (\ref{m_funk}) we immediately obtain the representation
for the Weyl function in the finite case:
\begin{equation}
\label{MN_repr} m^N(\lambda)=\widehat
v_1(\lambda)=\sum_{t=0}^\infty S_t(\lambda)r^N_t.
\end{equation}

Passing from dynamical systems with discrete time
(\ref{Jacobi_dyn}), (\ref{Jacobi_dyn_int}) to systems
(\ref{Syst_lambda_infinit}), (\ref{Syst_lambda_finit}) with  the
parameter $\lambda$ will be justified as soon as we show that sums
in (\ref{U_hat}) and (\ref{V_hat}) converge. This is shown in
\cite{MMS} for real matrices, this method without any changes
works here. Instead we show that the solution $\widehat
u_n(\lambda)\in l_2$ in some region. To this end we need estimates
on $|u^\delta_{n,t}|$. Introducing the notation
\begin{equation*}
M_t:={\rm max}_{1\leqslant n\leqslant
t}\left\{|u^\delta_{n,t}|,|u^\delta_{n,t-1}|\right\},
\end{equation*}
from the difference equation (the first line) of the system
(\ref{Jacobi_dyn}) we have the following estimate:
\begin{equation*}
|u^\delta_{n,t+1}|\leqslant (3B+1)M_t.
\end{equation*}
From this relation we obtain that
\begin{equation*}
M_{t+1}\leqslant (3B+1)M_t,\quad M_0=1.
\end{equation*}
Thus $M_t$ is bounded by the following:
\begin{equation*}
M_{t}\leqslant (3B+1)^t.
\end{equation*}
Then we note that since $u^\delta_{n,t}=0$ as soon as $t<n$, the
sum in (\ref{U_hat}) has the form
\begin{equation}
\label{U_hat1} \widehat u_n(\lambda):=\sum_{t=n}^\infty
S_t(\lambda)u^\delta_{n,t}
\end{equation}
Then we can estimate:
\begin{equation}
\label{U_hat11} |\widehat u_n(\lambda)|\leqslant |z(\lambda)|^n
M_n\sum_{t=0}^\infty |z(\lambda)|^t M_t.
\end{equation}
The sum in the right hand side of (\ref{U_hat1}) converges,
provided
\begin{equation*}
\left|z\left(\lambda\right)\right|^t(3B+1)^t<1,
\end{equation*}
or
\begin{equation}
\label{xi_est} \left|z(\lambda)\right|<\frac{1}{3B+1}.
\end{equation}
The estimate (\ref{xi_est}) is valid in the region $D\subset
\mathbb{C}_+$ specified in (\ref{D_reg}). We can then introduce
the notation $K(\lambda)=\sum_{t=0}^\infty |z(\lambda)|^t M_t,$
and get from (\ref{U_hat11}) that
\begin{equation*}
|\widehat u_n(\lambda)|\leqslant |z(\lambda)|^n
M_nK(\lambda),\quad \text{for} \quad \lambda\in D.
\end{equation*}
Then we immediately get that the $l_2-$norm of $\widehat
u(\lambda)$ can be estimated
\begin{equation}
\label{U_hat3} \sum_{n=0}^\infty|\widehat
u_n(\lambda)|^2\leqslant\ K(\lambda) \sum_{n=0}^\infty
|z(\lambda)|^{2n} M^2_n.
\end{equation}
We are left to note that the sum in the right hand side of
(\ref{U_hat3}) converges for $\lambda\in D$. The latter implies
that the solution (\ref{U_hat}) to (\ref{Syst_lambda_infinit}) is
from $l_2$ for $\lambda\in D$, and as such
\begin{equation*}
m(\lambda)=\widehat u_1(\lambda)=\sum_{t=0}^\infty
S_t(\lambda)r_t.
\end{equation*}
which completes the proof.

\end{proof}

\begin{remark}
The standard formula for the Weyl function in the real-valued case
(\ref{Weyl_real}) uses the spectral measure of $H$. In the
complex-valued case the obtained representation for the Weyl
function (\ref{M_inf}), (\ref{Resp_vect_N1}) in terms of some
measure constructed in the third Section. Thus measure was first
introduced in \cite{MM6} where it was used to solve the complex
moment problem. A substantive study of this measure and the
associated polynomials will be the subject of forthcoming
publications.
\end{remark}


\begin{thebibliography}{99}



\bibitem{AM}
{S.A. Avdonin, V.S. Mikhaylov,} \textit{The boundary control
approach to inverse spectral theory,} {Inverse Problems}, {\bf
26}, no. 4, 045009, 19 pp, 2010.



\bibitem{BB} {B. Beckermann} \textit{Complex Jacobi matrices,}
Journal of Computational and Applied Mathematics, 127, no. 1-2,
17--65, 2001, https://doi.org/10.1016/S0377-0427(00)00492-1

\bibitem{B07}
{M.I. Belishev}, {Recent progress in the boundary control method},
\textit{Inverse Problems}, {\bf 23}, R1, 2007.


\bibitem{BM_1}
{M.I.Belishev and V.S.Mikhailov}. \textit{Unified approach to
classical equations of inverse problem theory.} {Journal of
Inverse and Ill-Posed Problems}, 20, no 4, 461--488, 2012.


\bibitem{SG}
{F. Gesztesy, B. Simon}, \textit{$m-$functions and inverse
spectral analisys for dinite and semi-infinite Jacobi matrices},
J. d'Analyse Math. 73 (1997), 267-297




\bibitem{MM4}  {A. S. Mikhaylov, V. S. Mikhaylov,} Dynamic inverse
problem for complex Jacobi matrices. \textit{Zapiski sminarov
POMI,} 521, 136--154,  2023.  https://arxiv.org/abs/2505.20689

\bibitem{MM5}  {A.S. Mikhaylov, V.S. Mikhaylov,} Discrete dynamical systems:
inverse problems and related topics, \textit{Journal of Inverse
and Ill-posed Problems,} 2024, 10.1515/jiip-2024-0062

\bibitem{MM6}  {A.S. Mikhaylov, V.S. Mikhaylov,} On the complex moment
problem as a dynamic inverse problem for a discrete system,
https://arxiv.org/abs/2509.02443


\bibitem{MMS} {A. S. Mikhaylov, V. S. Mikhaylov, S.A. Simonov.} On the
relationship between Weyl functions of Jacobi matrices and
response vectors for special dynamical systems with discrete time,
Mathematical Methods in the Applied Sciences. 41, no. 16,
6401-6408, DOI: 10.1002/mma.5147, 2018.
https://arxiv.org/abs/1812.11041


\bibitem{W}
\text{https://en.wikipedia.org/wiki/Symmetric\_matrix}



\end{thebibliography}
\end{document}